\documentclass{amsart}
\usepackage{amsmath,amssymb,amscd,amsthm,amsfonts}
\usepackage[all,cmtip]{xy} 

\newcommand{\A}{\mathbb{A}}
 
 \newcommand{\C}{\mathbb{C}}
 \newcommand{\rO}{\mathrm{O}}
\newcommand{\cS}{\mathcal{S}}

\renewcommand{\c}[1]{\mathcal{#1}}

\renewcommand{\r}[1]{\mathrm{#1}}

\newcommand{\wtilde}{\widetilde}
\newcommand{\trpz}[1]{{{}^{\mathrm{t}}\negthinspace{#1}}}

\DeclareMathOperator{\GL}{\mathrm{GL}}

\DeclareMathOperator{\Sym}{\mathrm{Sym}}

\DeclareMathOperator{\Her}{\mathrm{Her}}
\DeclareMathOperator{\sHer}{\mathrm{sHer}}
\DeclareMathOperator{\Sp}{\mathrm{Sp}}
\DeclareMathOperator{\Mp}{\mathrm{Mp}}

\DeclareMathOperator{\tr}{\mathrm{tr}}
\DeclareMathOperator{\Ind}{\mathrm{Ind}}

\DeclareMathOperator{\Hom}{\mathrm{Hom}}
\DeclareMathOperator{\End}{\mathrm{End}}

\DeclareMathOperator{\Res}{\mathrm{Res}} 
\newcommand{\isom}{\cong}

\DeclareMathOperator{\lmod}{\backslash}
\let\Re\undefined
\DeclareMathOperator{\Re}{\mathrm{Re}}

\newcommand{\smatrix}[4]{\ensuremath\bigl( \begin{smallmatrix}
#1&#2\\ #3&#4
\end{smallmatrix} \bigr)}
\newcommand{\form}[2]{\langle{#1},{#2}\rangle}

\newtheorem{thm}{Theorem}[section]
\newtheorem{lemma}[thm]{Lemma}
\newtheorem{prop}[thm]{Proposition}
\newtheorem{cor}[thm]{Corollary} 
\theoremstyle{definition}
 
\theoremstyle{remark}

\newtheorem{rmk}[thm]{Remark}
\numberwithin{equation}{section}

\DeclareMathOperator{\Stab}{\mathrm{Stab}}
\newcommand{\til}{\tilde}
\usepackage{xy}
\usepackage{graphicx}
\begin{document}

\title{Irreducibility of Theta Lifting for Unitary Groups}
\author{Chenyan Wu}
\email{cywu@umn.edu}
\address{School of Mathematics, University of Minnesota,\\ 206 Church St. S.E., Minneapolis, MN 55455, USA}

\thanks{Corresponding author. Telephone number: +49(0)228-402-0.}
\keywords{regularised Siegel-Weil formula; first occurrence;  irreducibility of theta lift}


\date{}

\begin{abstract}
  This article shows that for unitary dual reductive pairs the first
  occurrence of theta lift of an irreducible cuspidal automorphic
  representation is irreducible. It also proves a refined tower
  property for theta lifts and the involutive property for twisted theta lifts.
\end{abstract}
\maketitle{}

\section*{Introduction}
\label{sec:intro}
This article studies the theta lifts between unitary groups. The main
result is that the `first occurrence' of theta lift is irreducible.
Let $k$ be a number field and $E$ a quadratic field extension of
$k$. Let $\A$ denote the adele ring of $k$. Fix a nontrivial additive
character $\psi$ of $k\lmod \A$. Let $X$ be a skew-Hermitian vector
space over $E$ and $Y$ a Hermitian vector space over $E$.  Let $G(X)$
(resp. $G(Y)$) denote the isometry group of $X$ (resp. $Y$). Fix a
pair of characters $\chi_1$ (resp. $\chi_2$) of $E^\times \lmod
\A_E^\times$ for the splitting of metaplectic cover over $G(X)$
(resp. $G(Y)$) (c.f. Sec. \ref{sec:split_meta}). Let $\pi$ be an irreducible
cuspidal automorphic representation of $G(X)$. Suppose that the theta
lift $\theta_{X,\psi}^Y(\pi)$ of $\pi$ from $G (X)$ to $G(Y)$ with
respect to $\psi$ and $( \chi_1,\chi_2)$ is nonzero and cuspidal. Then
we show that it is irreducible (Thm.~\ref{thm:theta_irred}). By the
tower property, here cuspidality is equivalent to saying that
$\theta_{X,\psi}^Y(\pi)$ is the first occurrence of the theta lift in
the Witt tower associated to $Y$. We also show an involution result
which is used in the proof of Thm.~\ref{thm:theta_irred}. Keep the
above assumption. Then the twisted theta lift
$\chi_1^{-1}\theta_{Y,\psi^{-1}}^X(\chi_2^{-1}\theta_{X,\psi}^Y(\pi))$
from $G(Y)$ back to $G(X)$ with respect to $\psi^{-1}$ and
$(\chi_1,\chi_2)$ is equal to $\pi$
(Thm. \ref{thm:theta_back_and_forth}). Analogous results hold for
twisted theta lift in the other direction.

This article is an extension of results of M\oe
glin\cite{MR1473166,MR1473165} and Jiang and Soudry\cite{MR2330445} to
the case of  unitary dual reductive pairs. M\oe glin\cite{MR1473166,MR1473165}
showed for the dual reductive pairs $\Sp_{2n}$ and $\rO(2m)$ the
irreducibility of first occurrence of theta lift in either
direction. She also proved the involutive property which is a key step
towards the proof of irreducibility. The odd orthogonal case was
treated by Jiang and Soudry\cite{MR2330445}. The groups involved are
the double cover $\wtilde{\Sp}(2n)$ of $\Sp(2n)$ and
$\rO(2m+1)$. Since in these cases the embedding of the dual reductive
pairs into the metaplectic group is canonical `non-twisted' theta lifts
are used.

Here we treat the dual reductive pair $G(X)$ and $G(Y)$ which are
unitary groups. The method of proof essentially follows M\oe glin and
Jiang and Soudry. However for unitary groups we need to treat the case
where neither $G(X)$ nor $G(Y)$ is quasi-split. Also because of the
non-uniqueness of splitting of metaplectic group over $G (Y)\times G
(X)$, we have to keep  track of the characters used to
determine a splitting and this results in the twist in theta lifts in our formula.

We sketch the idea of the proof and point out the difficulties. For a
non-negative integer $a$, form the skew-hermitian space $X_a$ by
adjoining $a$ hyperbolic planes $\ell_a=\ell_a^+ \oplus \ell_a^-$ to
$X$. First we show the involutive property. Via the regularised
Siegel-Weil formula for unitary groups due to Ichino\cite{MR2064052}
we show that for $A$ large enough the theta lift space
$\theta_{Y,\psi^{-1}}^{X_A}(\chi_2^{-1}\theta_{X,\psi}^Y(\pi))$ is
contained in a certain space of residues of Eisenstein series
(Prop.~\ref{prop:theta_is_Eis}). Here we need to be extra careful with
the choice of characters which determine the Weil representations
involved in constructing theta series. We use $(\chi_1,\chi_2)$ both
ways and $\psi$ and $\psi^{-1}$ for different directions of theta
lift. We want to remove the requirement that $A$ is large enough. Let
$a<A$. Let $Q_{A-a}$ be the parabolic subgroup of $G(X_A)$ stabilising
an $(A-a)$-dimensional isotropic subspace of $X_A$. Then we take
constant terms along $Q_{A-a}$ on the space of theta lift and the
space of Eisenstein series. On the theta side we expect to get
$\theta_{Y,\psi^{-1}}^{X_a}(\chi_2^{-1}\theta_{X,\psi}^Y(\pi))$ for
$a<A$. However we need a stronger tower property which is not known
for unitary groups. The result of
Rallis\cite{rallis84:_howe_dualit_conjec} on tower property does not
readily apply since it deals with symplectic and orthogonal groups
only and uses the property that symplectic group is always split. To
work around it we use mixed model for Weil representation, so unlike
the method in \cite{rallis84:_howe_dualit_conjec} we do not need to
completely linearise the Weil representation. Our global computation
is inspired by the local computation in \cite{MR1041060}. We are able
to generalise the tower property (Prop.~\ref{prop:tower_property}) and
the proof is simpler and more uniform than in
\cite{rallis84:_howe_dualit_conjec}. Thus we show that taking constant
term indeed gives the space
$\theta_{Y,\psi^{-1}}^{X_a}(\chi_2^{-1}\theta_{X,\psi}^Y(\pi))$. On
the other hand it can be shown that the constant term on the
Eisenstein side gives for $a>0$ residues of Eisenstein series, for
$a=0$ exactly $\chi_1\pi$ and for $a<0$ zero.  It should be mentioned
that we get a $\chi_1$-twist of $\pi$ because of the contribution from
determinants of unitary groups. Then the involutive property combined
with the tower property forces the irreducibility of
$\theta_{X,\psi}^Y(\pi)$.

The fundamental result of irreducibility of theta lifting for the
symplectic and orthogonal groups is used by Bergeron, Millson, and
M{\oe}glin in \cite{nicolas:_hodge} that gives Hodge type theorems on
Shimura varieties of orthogonal type. Thus our result has potential
application to the Shimura varieties of unitary type along the same
line.

\section{Notation}
\label{sec:notation}

Let $k$ be a number field 
and $E$ a quadratic extension of $k$. 
 Fix an element $\delta\in E$
 such that $\delta=-\overline {\delta}$. Let $\Delta = \delta^2$.
Let $\A$ be the adeles of
$k$. Let $X$ be a skew-Hermitian vector space of dimension $n$ over
$E$ with form $\form{\ }{\ }_X$ and $Y$ a Hermitian vector space of
dimension $m$ over $E$ with form $\form{\ }{\ }_Y$. Note that we
assume that $\form{\ }{\ }_X$ is linear in the first variable and
conjugate linear in the second variable whereas $\form{\ }{\ }_Y$ is
conjugate linear in the first variable and linear in the second
variable. Let $G(X)$ (resp. $G(Y)$) be the isometry group for $X$
(resp. $Y$). We let $G(X)$ act on the right and $G(Y)$ on the
left. Fix an additive character $\psi$ of $k\lmod \A$ and let
$\psi_E=\psi\circ \frac{1}{2}\tr_{E/k}$.  The $k$-vector space
$W=\Res_{E/k} ( Y\otimes_E X)$ endowed with the form
\begin{equation*}
  \form{y_1\otimes x_1}{y_2\otimes x_2}:= \tr_{E/k}\form{y_1}{y_2}_Y\overline{\form{x_1}{x_2}}_X
\end{equation*}
is symplectic. Sometimes we will drop $\Res _{E/k}$ to simplify
notation.  Let $\chi_1$ and $\chi_2$ be two characters of $E^\times
\lmod \A_E^\times$ such that $\chi_1|_{\A^\times} = \epsilon_{E/k}^m$
and $\chi_2|_{\A^\times} = \epsilon_{E/k}^n$ where $\epsilon_{E/k}$ is
the quadratic character of $k^\times \lmod \A^\times$ associated to
$E/k$ via Class Field Theory.  This is necessary if we need to
determine a splitting of the metaplectic group over the unitary dual
reductive pairs. Please see \cite{MR1286835} for more details. It
should be pointed out that $\chi_1$ (resp. $\chi_2$) is used to
determine an embedding of $G (X)$ (resp. $G (Y)$) into the metaplectic
group $\Mp (W)$ but the choice of $\chi_1$ (resp. $\chi_2$) is
restricted by the parity of the dimension of $G (Y)$ (resp. $G (X)$).

Let $X_a$ be the space formed by adjoining $a$ hyperbolic planes
$\ell_a=\ell_a^+ \oplus \ell_a^-$ to $X$. It is in the same Witt tower
as $X$. We may also define $X_{-b}$ if we can remove $b$ hyperbolic
planes from $X$, so $X=\ell_b^+ \oplus X_{-b} \oplus \ell_b^-$. Define
similarly $Y_a$ and $Y_{-b}$. Note that $\ell_a^\pm$ should conform to
the type of the original space. We hope this choice of notation does
not cause confusion. We will add in subscripts to indicate in which
space the $\ell_a^\pm$'s lie.

\section{Weil Representation}
\label{sec:weil_rep}

In this section we work in the local case, so temporarily we let $k$ denote a local field. 

\subsection{Representation of Metaplectic Group}
\label{sec:meta_rep}

We recall some results on the Weil representation. Because of the need to
describe the mixed model which figures prominently in the proof of
Rallis tower property for unitary groups, it is necessary that we
start with representations of the Heisenberg group. Then we describe
various models of Weil representation. References are
\cite{kudla:_notes_local_theta_corres} and \cite{MR1041060}.

Let $W=W^+ \oplus W^-$ be a symplectic space over $k$ with given
complete polarisation. One model of the representation of the
Heisenberg group $H(W)$ with central character $\psi$ is realised on
the space $S_{W^-}:=\Ind_{H(W^-)}^{H(W)}\psi$ where Schwartz induction
is used. We also let $\psi$ denote the character of $H(W^-):= W^-
\oplus k$ defined by $(w,t)\mapsto \psi(t)$. Let $\rho$ denote this
representation. For $g \in \Sp(W)$ define the representation $\rho^g$
by $\rho^g (h) = \rho (h^g)$ for $h\in H(W)$ where if $h = (w,t)$ then
$h^g = (wg,t)$. It also acts on $S_{W^-}$ with central character
$\psi$. By the Stone-von Neumann theorem,  there is an isomorphism between the  representations 
$\rho$ and $\rho^g$, which we now make explicit.

Consider the map $A^0(g) : S_{W^-} \rightarrow S_{W^-g^{-1}}$ given
by $(A^0(g)f)(h) := f(h^g)$. We have $\rho (h) A^0 (g) = A^0 (g) \rho
(h^g)$. In addition there exists an $H (W)$-intertwining isomorphism
$I_{W^-g^{-1},W^-}$ between $S_{W^-g^{-1}}$ and $S_{W^-}$ given by
\begin{equation}
  I_{W^-g^{-1},W^-} (f) (\cdot) = \int_{W^-g^{-1} \cap W^- \lmod W^-} f ((w^-,0)\cdot) dw^-
\end{equation}
where the choice of Haar measure is specified below. Define then $A (g) =
I_{W^-g^{-1},W^-}\circ A^0 (g)$. We use the unique choice of Haar
measure such that $A (g)$ is unitary.

Note that $A$ is not a representation of $\Sp(W)$, but $A$ lifts to a
representation, called the Weil representation, of
the metaplectic group $\Mp(W)$ which is a nontrivial $\C^1$-extension
of $\Sp(W)$. However if we restrict to a standard unipotent subgroup
of $\Sp(W)$, $A$ gives a representation(c.f. \cite{MR1286835}). 

The Schr\"odinger model of the Weil representation is realised on $\cS (W^+)$ which denotes the
Schwartz space of functions on $W^+$.  There is an isomorphism of
representations of $H(W)$
\begin{align*}
  S_{W^-} &\rightarrow \cS(W^+)\\
f &\mapsto \phi
\end{align*}
where $\phi$ is given by $\phi(w^+) = f(w^+,0)$ for $w^+ \in W^+$. To go
back, given $\phi$ then its preimage $f$ is given by 
\begin{align}
  f(w^++w^-,t) =
f((w^-,t+ \frac{1}{2}\form{w^+}{w^-}_W) (w^+,0)) =
\psi(t+\frac{1}{2}\form{w^+}{w^-}_W)\phi(w^+).
\end{align}
The action of $H(W)$ on $S(W^+)$ is given by\cite{kudla:_notes_local_theta_corres}
\begin{align}
  \rho((w^+ + w^-),t)\phi(w^+_0) = \psi(t+\form{w_0^+}{w^-}_W+\frac{1}{2}\form{w^+}{w^-}_W)\phi(w^+_0+w^+).
\end{align}

Let $g \in
\Sp(W)$ and write $g$ as $\smatrix{a}{b}{c}{d}$ with respect to the
polarisation. Note that $\Sp(W)$ acts on $W$ from the right, so
$a\in \End(W^+)$, $b\in \Hom(W^+,W^-)$, $c\in\Hom(W^-,W^+)$ and
$d\in\End(W^-)$. Then we transfer the operator $A (g)$ to 
$\cS(W^+)$ and denote it by $r(g)$. Then $r (g)$  is given by\cite{kudla:_notes_local_theta_corres}
\begin{multline}
 ( r(g)\phi)(w^+) = \int_{\ker(c)\lmod W^-}
 \psi\left(\frac{1}{2}\form{w^+a}{w^+b}+\form{w^-c}{w^+b}+\frac{1}{2}\form{w^-c}{w^-d}\right)\\
 \times\phi(w^+a+w^-c) d \mu_g w^-
\end{multline}
where we take the Haar measure on $\ker(c)\lmod W^-$ so that $r (g)$
becomes unitary. 
The obstruction to $r$ being a representation is given by the cocycle
$c_{W^-}(g_1,g_2)$ which is equal to $\gamma_k(\psi\circ
L(W^-,W^-g_2^{-1},W^-g_1))$, where $\gamma_k$ is the Weil index which
takes values in $8$-th roots of unity and $L$ gives the Leray
invariant. 
Please see \cite{rao93:_some_explic_formul_in_theor} for
the definitions of these. We lift $r$ to a representation of the
metaplectic group $\Mp(W)$ on $\cS (W^+)$ and denote it by
$\omega_{W,\psi}$. This is the Schr\"odinger model of the Weil
representation.

More precisely, $\Mp (W)$ is set-theoretically $\Sp (W)\times \C^1$ with multiplication given by
\begin{equation*}
  (g_1,z_1) (g_2,z_2)= (g_1g_2,z_1z_2c_{W^-}(g_1,g_2) )
\end{equation*}
and $r$ lifts to a representation of $\Mp (W)$ on $\cS (W^+)$ as follows
\begin{equation*}
  \omega_{W,\psi} (g,z) \phi = z \cdot r (g)\phi.
\end{equation*}
This definition of $\Mp (W)$ actually depends on $\psi$.

Now we note some properties\cite [pp. 36-37] {MR1041060} of the Weil representation. Let $W'$ be
the vector space with the same underlying vector space as $W$ but with
form $\form{\ }{\ }_{W'}= -\form{\ }{\ }_{W}$. We consider the impact
of changing $W$ to $W'$ and $\psi$ to $\psi_{-1}$ simultaneously.

\begin{prop}\label{prop:W_to_Wneg}
Let $j: \Sp(W) \rightarrow \Sp(W')$ be
the natural identification and extend it to 
\begin{align*}
  \til{\jmath}:\Mp(W)&\rightarrow \Mp(W')\\
(g,z) &\mapsto (j(g),z).
\end{align*}
Then  $\omega_{W,\psi} \isom \omega_{W',\psi_{-1}}\circ \til{\jmath}$.
\end{prop}

Another property that we will frequently invoke is the following

\begin{prop}\label{prop:direct_sum_of_weil_rep}
  Let $W_1$ and $W_2$ be two symplectic spaces and let $W=W_1\oplus W_2$. Let $j$ denote the natural embedding $\Sp(W_1)\times \Sp(W_2)\rightarrow \Sp(W)$ and extend it to the homomorphism
  \begin{align*}
\til{\jmath}: \Mp(W_1) \times \Mp(W_2) &\rightarrow \Mp(W) \\     
   ((g_1,z_1),(g_2,z_2)) &\mapsto (j(g_1,g_2),z_1z_2).
  \end{align*}
Then $\omega_{W,\psi}\circ \til{\jmath} \isom \omega_{W_1,\psi}\otimes \omega_{W_2,\psi}$ where $\otimes$ is replaced by completed tensor $\hat {\otimes}$ in the archimedean case.
\end{prop}

\subsection{Splitting of Metaplectic Cover}
\label{sec:split_meta}
Let $\eta=\psi_{1/2}$ as we use it often in Weil index. The metaplectic cover splits over the dual reductive pair $G(Y)\times G(X)$ since we are dealing with unitary groups. Given  a character $\chi_1$ of $E^\times$
such that $\chi_1|_{k^\times} = \epsilon_{E/k}^m$, there corresponds a splitting 
\begin{equation}
\xymatrix{
&   & \Mp(Y\otimes X) \ar[d]\\
&G(X) \ar@{.>}[ru]^{\iota_{\chi_1}} \ar[r] & \Sp(Y\otimes X)} 
\end{equation}
as in \cite{MR1286835,MR1327161}.  Of course $\iota_{\chi_1}$ also depends on $Y$. 

In the case where $X$ is split the splitting has an explicit description\cite [Theorem~3.1] {MR1286835}:
\begin{align}
  G (X) & \xrightarrow{\iota_{\chi_1}} \Mp (Y\otimes X) \\
  g &\mapsto (g, \beta_Y (g))
\end{align}
where $\beta_Y (g) = \chi_1 (x (g))\gamma_k (\eta\circ
RY)^{j(g)}$. The definitions of $x (g)$ and $j (g)$ can be found in
\cite [pp 370-371] {MR1286835} and $RY$ is the underlying $k$-vector
space of $Y$ equipped with the symmetric bilinear form $\frac{1}{2} \tr_{E/k}
\form{\ }{\ }_Y$.

  By \cite[Cor. A. 3]{MR1327161} the
splitting is compatible with taking direct sum. More precisely,
suppose that $X=X_1\oplus X_2$ is a direct sum of two skew-Hermitian spaces and use $\chi_1$ to determine  embeddings of $G (X)$, $G (X_1)$ and $G (X_2)$ into the corresponding metaplectic groups. Then we have a
commutative diagram:
\begin{equation}
  \xymatrix{
    &  G(X) \ar[r]^-{\iota_{\chi_1}}  & \Mp(Y\otimes X) \\
    &  G(X_1) \times G(X_2)\ar[u] \ar[r]^-{\iota_{\chi_1}\times\iota_{\chi_1}} & \Mp(Y\otimes X_1) \times  \Mp(Y\otimes X_2) \ar[u]^{\tilde{\jmath}}  } 
\end{equation}
where in the $\C^1$-part of the metaplectic groups $\tilde {\jmath}$ sends $( z_1 , z_2) $ to $z_1z_2$. 

Consider changing $\chi_1$ to some other character $\mu_1$ which satisfies $\mu_1|_{k^\times} =  \epsilon_{E/k}^m$. Let $\mu_1=\nu\chi_1$. Then $\nu$ is trivial on $k^\times$ and we construct a character $\nu' $ of $E^1$ by setting $\nu' (a/\overline {a}) = \nu(a)$. Then (c.f. \cite{MR1286835})
\begin{equation}\label{eq:change_character}
  \iota_{\mu_1} (g) = \nu' (\det g) \cdot \iota_{\chi_1} (g).
\end{equation}
It is useful to note that $\nu'^2 = \nu|_{E^1}$.

There is also an analogous version on the $Y$ side. Let $Y^\delta$ be
the skew-Hermitian space with the same underlying space as $Y$ but
with form $\form{y_1}{y_2 }_{Y^\delta} = \delta\form{y_2 }{y_1 }_Y$
and let $X^{\delta^{-1}}$ be the Hermitian space with the same
underlying space as $X$ but with form $\form{x_1}{x_2
}_{X^{\delta^{-1}}} = \delta^{-1}\form{x_2 }{x_1 }_X$. We have natural identifications $G
(X^{\delta^{-1} })=G (X)$, $G (Y^{\delta}) = G (Y)$ and $\Mp
(X^{\delta^{-1}}\otimes Y^{\delta}) = \Mp (Y\otimes X)$. Then the
roles of $G (X)$ and $G (Y)$ are completely symmetric.

We also note when $Y$ is split using $\chi_2$ such that $\chi_2|_{k^\times} = \epsilon_{E/k}^n$ we have explicit splitting
\begin{align}\label{eq:explicit_splitting_Y}
  G (Y) &\xrightarrow{\iota_{\chi_2}} \Mp (Y\otimes X) \\
  h &\mapsto (h, \beta_{X} (h))
\end{align}
where $\beta_{X} (h) = \chi_2 (x (h))\chi_2 (\delta)^{j
  (h)}\gamma_k (\eta\circ RX^{\delta^{-1}})^{j (h)}$ by \cite
[Theorem~3.1] {MR1286835}.

\subsection{Representation of Dual Reductive Pair}
\label{sec:rep_dual_pair}

We want to describe the Weil representation for dual reductive pairs
more explicitly, especially in the case where the skew-Hermitian space
is split. Suppose that $X$ is split with $\dim X =2n$, so
$G(X)=\r{U}(n,n)$. Choose a basis of $X$ such that the skew-Hermitian
form is given by $\smatrix{0}{1_n}{-1_n}{0}$. Let $m=\dim Y$.

We take $\chi_2$ to be the trivial character. This is allowed since $\dim X$ is
even. We claim that the corresponding splitting  is given by sending $h\in G (Y)$ to
$(h,1)\in \Mp (Y\otimes_E X)$. Assume first that $Y$ is split. Then we have explicit
description of splitting \eqref{eq:explicit_splitting_Y}. We need to
show that the quantity $\beta_X (h)$ for $h\in G (Y)$ is $1$. Let
$V=X^{\delta^{-1}}$ to simplify notation. Then as given in \cite
[Theorem~3.1] {MR1286835} we have
\begin{align}\label{eq:compute_beta}
  \beta_X (h) = \gamma_k (\eta\circ RV)^j 
  = (\Delta,\det V)_k\gamma_k (-\Delta,\eta)^{2n}\gamma_k (-1,\eta)^{-2n}
\end{align}
where $(\ ,\ )_k$ denotes the Hilbert symbol. Using properties of Weil index \cite [Lemma~4.1] {kudla:_notes_local_theta_corres}
we find \eqref{eq:compute_beta} is equal to
\begin{align}
  (\Delta,\det V)_k (-\Delta,-1)^{n}(-1,-1)^{-n} 
= (\Delta, \delta^{-2n})_k (\Delta,-1)^{n} 
=  (\Delta, -\Delta)^{n}_k
=1.
\end{align}
When $Y$ is non-split then splitting is determined via doubling. More precisely first we determine a splitting for $G (Y\oplus Y')$ and then restrict it to the subgroup $G (Y)$. Thus our claim is also true for non-split $Y$.

 Then the Weil representation of
$G(Y)\times G(X)$ is characterised by\cite{MR2064052}
\begin{equation}
  \begin{split}
      \label{eq:weil_rep}
  \omega\left(
    \begin{pmatrix}
      A &0 \\ 0 & \trpz{\bar{A}^{-1}}
    \end{pmatrix}\right)\phi(x) = &\chi_1(\det A)|\det A|_{E}^{\dim Y/2 } \phi(xA) \\
  \omega\left(
    \begin{pmatrix}
      1 &B \\ 0 & 1
    \end{pmatrix}\right)\phi(x) = &\phi(x)\psi_E(\frac{1}{2}\tr(\form{x}{x}_YB))\\
  \omega\left(
    \begin{pmatrix}
      0 &-1 \\ 1 & 0
    \end{pmatrix}\right)\phi(x) = &\gamma_k(\psi_{\frac{1}{2}}\circ RY)^{-n} \int_{Y^n(k)} \phi(y)\psi_E(\tr \form{y}{-x}_Y) dy\\
\omega(h)\phi(x)=&\phi(h^{-1}x).
  \end{split}
\end{equation}
Note that we regard $X$ and $Y$ as schemes defined over $k$. In
particular $Y(k)\isom E^{m}$. 

\subsection{Structure of Parabolic Subgroups}
\label{sec:structure_unipotent}

To prepare for the discussion of the mixed model we describe the structure
of parabolic subgroups. We do this in a coordinate free way and we
deviate a little from previous notation.

Let $E$ be a quadratic extension of $k$ or just $k$ itself. Now let
$W$ be an $\epsilon$-Hermitian space with $\epsilon=\pm 1$ to signify
whether the form is Hermitian or skew-Hermitian. Suppose $W$ can be
written as $\ell_a^+ \oplus W_0 \oplus \ell_a^-$. If the isometry
group acts on the right then we consider the parabolic subgroup
stabilising $\ell_a^-$. If the isometry group acts on the left then we
consider the parabolic subgroup stabilising $\ell_a^+$. Then the
parabolic subgroup of $G(W)$ has Levi part consisting of elements
\begin{align}
m(\alpha,\delta):=
  \begin{pmatrix}
    \alpha & &\\
           &\delta & \\
           &    & (\alpha^*)^{-1}
  \end{pmatrix}
\end{align}
where $\alpha \in \GL(\ell_a^+)$ and $\alpha^* \in \GL(\ell_a^-)$ is
the adjoint of $\alpha$ and $\delta \in G(W_0)$. Adjoints are taken with respect to the pairing $\form {\ } {\ }_W$.
 If the action is on the right then it has unipotent subgroup consisting of elements of the form
\begin{align}\label{eq:def_n_mu_beta}
n(\mu,\beta):=
  \begin{pmatrix}
    1 & \mu & \beta - \mu^*\mu/2\\
           &1 & -\mu^*\\
           &    & 1
  \end{pmatrix}
\end{align}
where $\mu \in \Hom_E(\ell_a^+,W_0)$ and $\beta \in
\Hom_E(\ell_a^+,\ell_a^-)$ such that $\beta+\beta^*=0$. 
If the action is on the left then it has unipotent subgroup consisting of elements of the form
\begin{align}
n(\mu,\beta):=
  \begin{pmatrix}
    1 & \mu & \beta - \mu\mu^*/2\\
           &1 & -\mu^*\\
           &    & 1
  \end{pmatrix}
\end{align}
where $\mu \in \Hom_E(W_0,\ell_a^+)$ and $\beta \in
\Hom_E(\ell_a^-,\ell_a^+)$ such that $\beta+\beta^*=0$. 

In particular after taking dual standard basis if $W$ is symplectic then $\beta$ can be regarded as in
$\Sym_a$; if $W$ is Hermitian then $\beta$ can be regarded as in
$\sHer_a$; if $W$ is skew-Hermitian then $\beta$ can be regarded as in
$\Her_a$.

\subsection{Mixed Model}
\label{sec:mixed_model}

We are more concerned with the action of the unipotent subgroups on
the mixed model. Note that the splitting is trivial over our unipotent subgroups.
 For all other expressions that we will need we will
invariably reduce via Prop.~\ref{prop:direct_sum_of_weil_rep} to the split
case. This is key in the demonstration of tower property.

Let $V=V^+\oplus V^-$ be a symplectic space.  Let $W_0$ be another
symplectic space and consider the symplectic space $W=V\oplus
W_0$. Thus we are adding hyperbolic planes to $W_0$. Assume we have
polarisation $W_0=W_0^+ \oplus W_0^-$. Suppose the
representation of the Heisenberg group $H(W)$ is realised on
$S_{W_0^-\oplus V^-}$.  Recall the element $n(\mu,\beta)$ defined in
\eqref{eq:def_n_mu_beta}. We compute the action of the operator
$A(n(\mu,0))$ on $f_0\otimes f \in S_{W_0^-} \otimes S_{V^-}$. Note that
$n(\mu,0)$ preserves $W_0^-\oplus V^-$ and thus there is no need to apply
`$I_{W^-g^{-1},W^-}$'. We compute:
\begin{align*}
  &A(n(\mu,0))(f_0 \otimes f) ((w_0,t), (v^+  ,0)) \\
  =& f_0((w_0+ \mu(v^+),t))f( (v^+ - \mu^*(w_0)-\frac{1}{2}\mu^* \mu(v^+),0))\\
  =& f_0((w_0,t)(\mu(v^+),-\frac{1}{2}\form{w_0}{\mu(v^+)}_{W_0})) \\
  &\qquad\times f( (- \mu^*(w_0)-\frac{1}{2}\mu^* \mu(v^+),\frac{1}{2}\form{ \mu^*(w_0)+\frac{1}{2}\mu^* \mu(v^+)}{v^+}_{V})(v^+,0))\\
  = &f_0((w_0,t)(\mu(v^+),0))\\
&  \times f( (- \mu^*(w_0)-\frac{1}{2}\mu^* \mu(v^+),\frac{1}{2}\form{ \mu^*(w_0)+\frac{1}{2}\mu^* \mu(v^+)}{v^+}_{V}-\frac{1}{2}\form{w_0}{\mu(v^+)}_{W_0})(v^+,0))\\
=&f_0((w_0,t)(\mu(v^+),0))
  f( (- \mu^*(w_0)-\frac{1}{2}\mu^* \mu(v^+),0)(v^+,0))\\
 =&\rho_0((\mu(v^+),0))f_0((w_0,t)) \cdot  f(v^+,0)
\end{align*}
where $\rho_0$ is the representation of $H(W_0)$ on $S_{W_0^-}$.
We transfer the action to the model $S_{W_0^-}\otimes \cS(V^+)$ to get
\begin{equation}\label{eq:action_n_mu_zero}
   r(n(\mu,0))(f_0 \otimes \phi) ((w_0,t), v^+ ) = \rho_0((\mu(v^+),0))f_0((w_0,t)) \cdot  \phi(v^+).
\end{equation}
This is unitary.

Now we describe the Weil representation $\omega_{X,Y}$ of $G(Y) \times G(X)$ on the mixed
model for certain unipotent elements. We do not assume that $X$ or $Y$ is split. Suppose that $X$ can
be written as $X= \ell_a^+ \oplus X_0 \oplus \ell_a^-$. Here $X_0$ is
not necessarily anisotropic. Let $S_0$ be any model for the Weil
representation $G(Y)\times G(X_0)$. Let the Weil representation of
$G(Y)\times G(\ell_a^+\oplus \ell_a^-)$ be realised on $\cS(Y\otimes
\ell_a^+)$ as in Sec. \ref{sec:rep_dual_pair}. Note that $Y\otimes
\ell_a^+ \isom Y^a$. We describe the action of $n(\mu,0)\in G(X)$ on
the model $S_0\otimes \cS(Y\otimes \ell_a^+)$, where $\mu \in
\Hom_E(\ell_a^+,X_0)$. In the archimedean version we replace tensor by
completed tensor. This is the special case where $W=\Res_{E/k} (Y\otimes_E X)$ and thus $V^\pm
= \Res_{E/k} (Y\otimes \ell_a^\pm)$ and $W_0 = \Res_{E/k} (Y\otimes X_0)$.
Thus by \eqref{eq:action_n_mu_zero} we find for $\phi_0\otimes \phi \in S_0\otimes \cS (Y\otimes \ell_a^+)$
\begin{equation}
  \label{eq:action_n_mu_zero_dual_red}
  \omega_{X,Y}(n(\mu,0),1_{G(Y)}) (\phi_0\otimes
  \phi)(\cdot, y) = \rho_0((\mathbf{1}_Y\otimes \mu(y),0))\phi_0(\cdot)
\times\phi(y)
\end{equation}
where $\rho_0$ is the representation of $H (Y\otimes X_0)$ on $S_0$. 

\section{Tower Property}
\label{sec:tower_property}


Let $\pi$ be an irreducible cuspidal automorphic representation of
$G(Y)$. Let $W^+$ be some maximal isotropic subspace of
$W=\Res_{E/k}Y\otimes_E X$. Let the Weil representation $\omega$ of
$G(Y)\times G(X)$ be realised on the Schwartz space $\cS (W^+)$ with
respect to $\psi$ and $(\chi_1,\chi_2)$. For $f\in \pi$ and $\phi \in \cS (W^+)$
we define
\begin{equation}
  \theta(g,\phi,f)= \int_{[G(Y)]} \theta_{X,Y}(g,h,\phi) f(h) dh,
\end{equation}
where $\theta_{X,Y}(g,h,\phi) = \sum_{w\in W^+ (k)}\omega (g,h)\phi
(w)$. When there is no need to emphasise the chosen maximal isotropic
subspace $W^+$ we write  $(Y\otimes X)^+$ for some
maximal isotropic subspace of $\Res_{E/k}Y\otimes_E X$ and the space $\cS (W^+)$ will be denoted as
$\cS_{X,Y}$. The global theta lift $\theta_Y^X(\pi)$ of $\pi$ from
$G(Y)$ to $G(X)$ is defined to be the space generated by all such
functions $\theta(g,\phi,f)$. We consider the tower of theta lifts to
$G(X_a)$ for varying $a$'s. We can also lift from $G(X)$ to $G(Y_a)$
for $\pi$ an irreducible cuspidal automorphic representation of
$G(Y)$. Note that as the Weil representation depends on the additive
character $\psi$, the theta lifts depend on $\psi$, but we suppress it
from notation in this section. We have also suppressed the dependency
on $\chi_1$ and $\chi_2$. Since the dimensions of $X_a$ and $X$ are of the same
parity we can use the same $\chi_2$ for the splittings over $G
(Y)$. Implicitly we also need an embedding of $G (Y) \rightarrow \Mp
(Y\otimes \ell_a)$ and  the trivial character is used to determine
it.


We will extend the computation in \cite{rallis84:_howe_dualit_conjec}
to the unitary case. Our computation is simpler and more uniform. It
follows the spirit of the local computation of covariants in
\cite[Chap. 3. V]{MR1041060}. Since the isometry group of a Hermitian
space can also be regarded as the isometry group of a skew-Hermitian
space via a (non-canonical) isomorphism which depends on $\delta\in
E$ (c.f. Sec.~\ref{sec:split_meta}), we just need to demonstrate the tower property in one
direction. 
Let $Q_a$ be the parabolic subgroup of $G(X_a)$ stabilising $\ell_a^-$
which is an isotropic subspace of $X_a$. Similarly define $R_a$ to be
the parabolic subgroup of $G(Y_a)$ stabilising $\ell_a^+$, an
isotropic subspace of $Y_a$.

\begin{prop}\label{prop:tower_property}
Let $\pi$ be an irreducible cuspidal automorphic representation of $G(Y)$. 
Let $\phi \in \cS((Y\otimes X)^+\oplus Y^a)(\A)$. Then for $f\in \pi$ and $g\in G(X)(\A)$, the constant term along $Q_a$ of the theta lift is equal to
    \begin{equation}
      [\theta_{X_a,Y}(g,\phi,f)]_{Q_a} = \theta_{X,Y}(g,\phi(\cdot,0),f).
    \end{equation}
    Hence as $G (X)$-representations
\begin{equation}
\Res_{G(X)}(\theta_{Y}^{X_a}(\pi)_{Q_a}) = \theta_{Y}^{X}(\pi).
     \end{equation}
The action of $m(\GL_a (\A_E) ,1)$ on $\theta_{Y}^{X_a}(\pi)_{Q_a}$ is given by the character
\begin{equation}
  \chi_1\circ \det |\det |_{\A_E}^{\dim Y /2}.
\end{equation}
There is also a symmetric version for lifts from $G(X)$ to $G(Y_a)$.
\end{prop}

\begin{proof}
To make notation less cluttered we omit writing $(k)$ for taking rational points.
We need to compute
\begin{equation}\label{eq:constant_term}
   \int_{[U_{Q_a}]}\int_{[G(Y)]} \theta_{X_a,Y}(ng,h,\phi)f(h)dhdn.
\end{equation}

Step 1. Consider first 
\begin{equation}\label{eq:int_UQa}
  \int_{[U_{Q_a}]} \theta_{X_a,Y}(ng,h,\phi) dn.
\end{equation}
The exchange of order of integration will be justified at the end of the computation.
Recalling the structure of $U_{Q_a}$ (c.f. \eqref{eq:def_n_mu_beta}) we find that \eqref{eq:int_UQa} is equal to
\begin{align}
&\int_{[\Hom_E(\ell_a^+, X)]}\int_{\beta \in [\Her_a]} \theta_{X_a,Y}(n(\mu,\beta)g,h,\phi) d\beta d\mu\\
  =&\int_{[\Hom_E(\ell_a^+, X)]}\int_{\beta \in [\Her_a]} \sum_{z \in
    (Y\otimes X)^+}\sum_{y \in Y^a}
  \omega_{X_a,Y}(n(\mu,\beta)g,h)\phi(z,y) d\beta d\mu.
\end{align}
  Without loss
of generality we assume $\phi = \phi_0 \otimes \phi_1$ with $\phi_0
\in S_{X,Y}$ and $\phi_1 \in \cS(Y\otimes \ell_a^+)$. By
\eqref{eq:weil_rep} we write out the action of  $\beta$ to get
\begin{align}
   &\int_{[\Hom_E(\ell_a^+, X)]}\int_{\beta \in [\Her_a]} \sum_{z \in
    (Y\otimes X)^+}\sum_{y \in Y^a}
  \omega_{X_a,Y}(n(\mu,0)g,h)\phi_0\otimes\phi_1(z,y) \psi_E(\tr(\form{y}{y}_Y\beta))d\beta d\mu
\end{align}
so the integration vanishes unless $\form{y}{y}_Y=0$ and we get
\begin{align}\label{eq:r=0_refer_back}
  &\int_{[\Hom_E(\ell_a^+, X)]} \sum_{z \in
     (Y\otimes X)^+}\sum_{\substack{y \in Y^a\\ \form{y}{y}_Y=0 }}
  \omega_{X_a,Y}(n(\mu,0)g,h)\phi_0\otimes\phi_1(z,y) d\mu.
\end{align}
The condition $\form{y}{y}_Y=0$ means that the columns of $y\in Y^a$ span an
isotropic subspace of $Y$. We have identified $Y\otimes \ell_a^+$ with $Y^a$.

Step 2. We will show that only $y=0$ contributes to the above integral. Note that
if $Y$ is anisotropic we are done now. Decompose $Y$ as
$\ell_b^+\oplus Y_0 \oplus \ell_b^-$ with $Y_0$ anisotropic. Choose  dual bases for $\ell_b^+$ and   $\ell_b^-$
and denote the elements by $e_1,e_2,\ldots,e_b$ and $e_{-1},e_{-2},\ldots,e_{-b}$ respectively. 
We consider the orbits of $y$ under the action of $G(Y) (k)\times \GL_a (E)$. To alleviate notation $(k)$ and $(E)$ are dropped afterwards. We will split the integral \eqref{eq:r=0_refer_back} according to the orbits and analyse them one  at a time.
The orbits are parametrised by the rank $r$ of $y$
and we can choose 
\begin{equation}
  \label{eq:def_yr}
  y_r:=(e_1,e_2,\ldots,e_r,0,\ldots, 0)
\end{equation}
 as representatives
of the orbits for $r$ running from $0$ to $\min (a,b)$.

Then \eqref{eq:r=0_refer_back} is equal to
 \begin{align}
   & \sum_{r=0}^{\min(a,b)} \int_{[\Hom_E(\ell_a^+, X)]} \sum_{z \in
    (Y\otimes X)^+} \sum_{(\gamma,\nu) }
  \omega_{X_a,Y}(n(\mu,0)g,h)\phi_0\otimes\phi_1(z,\gamma^{-1} y_r \nu) d\mu
 \end{align}
 where $(\gamma,\nu)$ runs over $ \Stab_{G(Y)\times \GL_a} y_r\lmod
 (G(Y)\times \GL_a)$. By the explicit description of Weil representation the above is equal to
 \begin{align}
   & \sum_{r=0}^{\min(a,b)} \int_{[\Hom_E(\ell_a^+, X)]} \sum_{z \in
    (Y\otimes X)^+} \sum_{(\gamma,\nu) }
  \omega_{X_a,Y}(m(\nu,\mathbf{1})n(\mu,0)g,\gamma h)\phi_0\otimes\phi_1(z, y_r) d\mu\\
\label{eq:int_orbits}
   =& \sum_{r=0}^{\min(a,b)} \int_{[\Hom_E(\ell_a^+, X)]} \sum_{z \in
    (Y\otimes X)^+} \sum_{(\gamma,\nu)}
  \omega_{X_a,Y}(n(\mu,0)m(\nu,\mathbf{1})g,\gamma h)\phi_0\otimes\phi_1(z, y_r) d\mu.
 \end{align}
The above equality holds because
 \begin{equation*}
   \sum_{z \in (Y\otimes X)^+}
 \omega_{X,Y}(g,\gamma h)\phi_0(z) = \sum_{z \in (Y\otimes X)^+}
 \omega_{X,Y}(g, h)\phi_0(z).
 \end{equation*}
  and because exchanging
 $m(\nu,\mathbf{1})$ and $n(\mu,0)$ does not change the value of the
 integral.

Step 3. Now we focus on the subintegral of \eqref{eq:int_orbits} for a fixed $r \ge 1$.
Using \eqref{eq:action_n_mu_zero_dual_red}, for each $r$ we get:
\begin{align}
&\int_{[\Hom_E(\ell_a^+, X)]} \sum_{z \in    (Y\otimes X)^+} 
\sum_{(\gamma,\nu) }
\omega_{X_a,Y}(n(\mu,0)m(\nu,\mathbf{1})g,\gamma h) (\phi_0\otimes\phi_1)(z, y_r) d\mu\\
\label{eq:integrating_mu}
=&\int_{[\Hom_E(\ell_a^+, X)]} \sum_{z \in    (Y\otimes X)^+} 
\sum_{(\gamma,\nu) }
\rho_0(1_Y\otimes \mu(y_r),0)\omega_{X_a,Y}(m(\nu,\mathbf{1})g,\gamma h) (\phi_0\otimes\phi_1)(z, y_r) d\mu.
\end{align}
The representation $\rho_0$ of the Heisenberg group $H(Y\otimes X)$
acts only on the $\cS((Y\otimes X)^+)$-part. Let $\ell_r^\pm$ denote
the span of columns of $y_r$. These are subspaces of $\ell_b^\pm$. Let $Y_{-r}$ denote the orthogonal
complement of $\ell_r^\pm$ in $Y$. Now we
use a more concrete model $\cS((Y_{-r}\otimes X)^+ \oplus
\ell_r^-\otimes X)$ for $\rho_0$. For $\phi_{00}\in \cS((Y_{-r}\otimes
X)^+)$ and $\phi_r\in \cS(\ell_r^-\otimes X)$, $\rho_0$ acts as
\begin{align}
   \rho_0(1_Y\otimes \mu(y_r),0) (\phi_{00}\otimes \phi_r)(z_0,x) =& \rho_r(1_{\ell_r}\otimes \mu(y_r),0)\phi_r(x) \cdot \phi_{00}(z_0)\\
=& \psi_E( \tr(\form{x}{\mu(y_r)}_X)) \phi_r(x)\phi_{00}(z_0)
\end{align}
where $\rho_r$ is the representation of the Heisenberg group $H(\ell_r\otimes X)$. The
above equation holds because the action of $1_Y\otimes \mu(y_r)$ concentrates in the $\phi_r$-part
as $1_{\ell_r}\otimes \mu(y_r)$ and in the final line we let
$\mu$ acts on $y_r$ row-wise by abuse of notation. Applying this the
integral \eqref{eq:integrating_mu} becomes
\begin{multline}\label{eq:int_against_mu}
  \int_{[\Hom_E(\ell_a^+, X)]} \sum_{z_0 \in    (Y_{-r}\otimes X)^+} \sum_{x\in \ell_r^-\otimes X}
\sum_{(\gamma,\nu) } \\
\psi_E(\tr\form{x}{\mu(y_r)}_X)\omega_{X_a,Y}(m(\nu,\mathbf{1})g,\gamma h) (\phi_0\otimes\phi_1)(z_0,x, y_r) d\mu.
\end{multline}
Thus for a fixed $x$ the integration against $\mu$ vanishes unless the space spanned
by the rows of $x$ is orthogonal to the space spanned by the rows of
$\mu(y_r)$ for all $\mu$. Note that $y_r \in \ell_{r,Y}^+\otimes \ell_{a,X}^+$ and thus $\mu(y_r)\in \ell_r^+\otimes X$. 
 Since we are in the case $r\ge 1$  the space spanned by the rows of
$\mu(y_r)$ for all $\mu$ is the whole of $X$. Thus only $x=0$ contributes and \eqref{eq:int_against_mu} becomes
\begin{align}
   \sum_{z_0 \in    (Y_{-r}\otimes X)^+} 
  \sum_{(\gamma,\nu) } \omega_{X_a,Y}(m(\nu,\mathbf{1})g,\gamma h) (\phi_0\otimes\phi_1)(z_0,0, y_r) .
\end{align}

Step 4. Still in the case $r\ge 1$, we multiply the above by $f(h)$ and integrate over $h$. Note that the sum
over $(\gamma,\nu)\in\Stab_{G(Y)\times \GL_a} y_r\lmod (G(Y)\times
\GL_a)$ can be written as the double sum over $\nu \in S_1:=\Stab_{\GL_a}y_r
\lmod \GL_a$ and $\gamma$ in the stabiliser $S_2$ in $G(Y)$ of the $\GL_a$-orbit of $y_r$. Obviously this
contains the unipotent subgroup $U_{R_r}$ of $G(Y)$ for $r\ge 1$. Thus  we get
\begin{align*}
  &\int_{[G(Y)]} \sum_{z_0 \in (Y_{-r}\otimes X)^+}
  \sum_{(\gamma,\nu) } \omega_{X_a,Y}(m(\nu,\mathbf{1})g,\gamma
  h) (\phi_0\otimes\phi_1)(z_0,0, y_r) f(h) dh\\  
  =& \int_{S_2 (k)\lmod G(Y)(\A)} \sum_{z_0 \in (Y_{-r}\otimes X)^+}
  \sum_{\nu
    \in \Stab_{\GL_a}y_r \lmod \GL_a } \\
  &\qquad \omega_{X_a,Y}(m(\nu,\mathbf{1})g, h) (\phi_0\otimes\phi_1)(z_0,0,
  y_r) f(h) dh.
\end{align*}

The expression
\begin{align}
  \begin{split}
     h \mapsto \sum_{z_0 \in (Y_{-r}\otimes X)^+}
   \sum_{\nu \in
    \Stab_{\GL_a}y_r \lmod \GL_a } 
  \omega_{X_a,Y}(m(\nu,\mathbf{1})g, h) (\phi_0\otimes\phi_1)(z_0,0, y_r) 
  \end{split}
\end{align}
is by the following lemma invariant under all $n(\mu',\beta')\in
U_{R_r}(\A)$ for $\mu' \in \Hom_E(\ell_r^+, Y_{-r})$ and $\beta' \in
\Hom_E(\ell_r^+,\ell_r^-)$ such that $\beta' {}^*=-\beta'$.  Thus we can
decompose the integration over $h$ to get an inner integral
\begin{align}
  \int_{[U_{R_r}]}f(nh)dn
\end{align}
and this vanishes because $f$ is cuspidal.  Thus only for $r=0$ does
the subintegral not necessarily vanish. 

Step 5. Consider the case $r=0$ i.e. the orbit containing the single element $y=0$. 
Setting $y=0$ in \eqref{eq:r=0_refer_back} and restricting to $g\in G(X)(\A)$ we are left
with
\begin{align}
  &\int_{[\Hom_E(\ell_a^+, X)]} \sum_{z \in    (Y\otimes X)^+} 
  \omega_{X_a,Y}(n(\mu,0)g,h) (\phi_0\otimes\phi_1)(z,0) d\mu\\
  =& \sum_{z \in    (Y\otimes X)^+}   \omega_{X_a,Y}(g,h) (\phi_0\otimes\phi_1)(z,0) \\
=& \theta_{X,Y}(g,h,\phi_0\otimes\phi_1(\cdot,0)).
\end{align}
Finally we integrate over $[G (Y)]$ against $f (h)$ to get $\theta_{X,Y}(g,\phi_0\otimes\phi_1(\cdot,0),f)$ as required. The change of order in integration is justified by the absolute convergence of the integral. 

The action of $m (\GL_a (\A_E),1)$ follows simply from the explicit formulae of Weil representation.
\end{proof}

\begin{lemma}
  Let $a\ge 1$ and $b\ge 1$. Assume $X=\ell_{a,X}^+\oplus X_{-a}
  \oplus \ell_{a,X}^-$ and $Y=\ell_{b,Y}^+\oplus Y_{-b} \oplus
  \ell_{b,Y}^-$. Suppose the Weil representation $\omega_{X,Y}$ of
  $G(Y) \times G(X)$ is realised on $S:=\cS( (Y_{-b} \otimes X_{-a})^+
  \oplus (\ell_{b,Y}^- \otimes X_{-a}) \oplus (Y\otimes
  \ell_{a,X}^+))$. Then for $\phi \in S$
  \begin{equation}
    \omega_{X,Y}(1,n_Y(\mu,\beta))\phi(z_0,0,y_b) =  \phi  (z_0,0,y_b)
  \end{equation}
where $y_b$ is as in \eqref{eq:def_yr}.
\end{lemma}
\begin{proof}
  We only need to show the equality for $\phi= \phi_1\otimes \phi_2
  \otimes \phi_3$ for $\phi_1 \in \cS((Y_{-b} \otimes X_{-a})^+)$, $\phi_2
  \in \cS(\ell_{b,Y}^- \otimes X_{-a})$ and $\phi_3 \in (Y\otimes \ell_{a,X}^+)$.
We compute 
\begin{align}
  &  \omega_{X,Y}(1,n_Y(0,\beta))\phi_1\otimes \phi_2  \otimes \phi_3(z_0,0,y_b)\\
  =& \phi_1(z_0)
  \omega_{X_{-a},\ell_{b,Y}}(1,n_Y(0,\beta))\phi_2(0)
  \omega_{\ell_{a,X},Y}(1,n_Y(0,\beta))
  \phi_3(y_b).
\end{align}
By \eqref{eq:weil_rep} we find the above is equal to
$ \phi_1\otimes \phi_2  \otimes \phi_3(z_0,0,y_b)$.

Next we compute
\begin{align}
  &  \omega_{X,Y}(1,n_Y(\mu,0))\phi_1\otimes \phi_2  \otimes \phi_3(z_0,0,y_b)\\
  =& \omega_{X_{-a},Y}(1,n_Y(\mu,0))\phi_1\otimes \phi_2(z_0,0) \omega_{\ell_{a,X},Y}(1,n_Y(\mu,0))\phi_3(y_b)
\end{align}
By \eqref{eq:action_n_mu_zero_dual_red} and \eqref{eq:weil_rep} we get also $ \phi_1\otimes \phi_2  \otimes \phi_3(z_0,0,y_b)$.

\end{proof}

\section{Regularised Siegel-Weil Formula}
\label{sec:reg_sw}

\providecommand{\rU}{\r{U}} 
\providecommand{\REG}{\r{REG}}
\providecommand{\abc}{\r{abc}}

To relate the theta lift space to Eisenstein series we recall the regularised Siegel-Weil
Formula for unitary groups from \cite{MR2064052}. We deviate from our usual notation.

Let $V$ be a Hermitian space of dimension $m$. Let $H=\rU(V)$ and
$G=\rU(n,n)$. Suppose $n<m\le 2n$ and $m-r \le n$. Then we define the
complementary space $V^c$ of $V$ as follows. Let $m^c=\dim V^c$. Then
$m+m^c=2n$ and $V^c$ is required to be in the same Witt tower as
$V$. Fix $K:=K_H$ a maximal compact subgroup of $H(\A)$ and $K_G$ a
maximal compact subgroup of $G(\A)$. In this section we use the
trivial character to split the metaplectic group over $H (\A)$ and  $\chi_1$   to
split the metaplectic group over $G (\A)$.

We consider the theta integral
\begin{equation}
  I(g,\phi) := \int_{[H]} \theta(g,h,\phi)dh.
\end{equation}
This may not be absolutely convergent. Let $\cS(V^n(\A))_\abc$ denote
the subspace of $\cS(V^n(\A))$ consisting of function $\phi$ such that
$I(g,\phi)$ is absolutely convergent for all $g$. This space is
nonempty. Then $I$ defines an $H(\A)$-invariant map
\begin{equation*}
  I: \cS(V^n(\A))_\abc \rightarrow \c{A}^\infty(G)
\end{equation*}
where $\c{A}^\infty(G)$ is the space of smooth automorphic forms on
$G(\A)$  without the
$K_G$-finiteness condition. Then Ichino\cite{MR2064052} showed that there exists a canonical extension of $I$:
\begin{prop}
  Assume $m\le n$. Then there exists a unique $H(\A)$-invariant
  extension $I_{\REG}$  of $I$ to $\cS(V^n(\A))$. More precisely, it is realised as
  \begin{equation}
    c_\alpha^{-1} \int_{[H]}\theta(g,h,\omega(\alpha)\phi) dh
  \end{equation}
where 
$\alpha$  is a suitable element in the Hecke algebra of
$G$. It can also be taken to be an element in the Hecke algebra of
$H$ which acts on the trivial representation of $H$ by the scalar
$c_\alpha$. 
\end{prop}

Now to distinguish the regularised theta integral associated to
different groups we add in subscripts, so $I_{V,\REG}$ is what we call
$I_{\REG}$ above.

We also need the definition of the Siegel-Eisenstein series. First we
define the Siegel-Weil section associated to $\Phi\in \cS(V^n(\A))$. Let
$P$ be the Siegel parabolic subgroup of $G$ and  $N$ the unipotent part.
For $g\in G(\A)$ decompose $g$ as
$g=m(A)nk$ with $A\in\Res_{E/k} \GL_n(\A)$, $n\in N(\A)$ and $k\in K_G$. Set
$a(g) =\det A$ in any such decomposition of $g$ and then the quantity
$|a(g)|$ is well-defined.  The Siegel-Weil section associated to
$\Phi\in\cS(U^n(\A))$ is defined to be
\begin{equation*}
  F_\Phi (g,s) = | a(g)|^{s-s_0}\omega(g)\Phi(0),
\end{equation*}
where $s_0=(m-n)/2$. This is a section in the induced representation
$\Ind_{P(\A)}^{G(\A)} \chi_1|\; |^{s}$.

It is useful to also note the
local version.  We let $R_n(V_v)$ denote the set of sections $g\mapsto
\omega(g)\Phi(0)$ inside $\Ind_{P_v}^{G_v} \chi_1|\; |^{s_0} $.

Returning to the global case, form the
Siegel-Eisenstein series
\begin{equation}
  E(g,s,F_\Phi) = \sum_{\gamma\in P(k)\lmod G(k)} F_\Phi(\gamma g, s).
\end{equation}
It is absolutely convergent for $\Re s >>0$. For   a
$K_G$-finite element $\Phi$ in $\cS(V^n(\A))$,  $E(g,s,F_\Phi)$ has meromorphic continuation
to the whole complex plane.

Then  the regularised Siegel-Weil formula says
\begin{thm}\emph{\cite[Thm.~4.1]{MR2064052}}
  Suppose $m >n$. Let $\Phi$ be a $K_G$-finite element in
  $\cS(V^n(\A))$. Then
  \begin{equation}
    \label{eq:reg_sw}
    \Res_{s=\frac{m-n}{2}}E(g,s,F_\Phi) = c_K I_{V^c,\REG}(g,\pi_V^{V^c}\pi_K\Phi).
  \end{equation}
\end{thm}
\begin{rmk}
  Here $c_K$ is a constant depending only on $K$. For the definition
  of $\pi_V^{V^c}\pi_K$ please see \eqref{eq:defn_pi_QQ_pi_K}. It
  sends $\Phi$ to a function in $\cS(V^{c,n}(\A))$.  
\end{rmk}
\begin{cor}
\label{cor:SW}
  Suppose $m^c =\dim V^c <n$. For any  $K_G$-finite element $\Phi^c$
  in $\cS((V^c)^n(\A))$, there exists a $K_G$-finite element $\Phi$ in
  $\cS(V^n(\A))$ such that the following holds:
  \begin{equation}
    \label{eq:reg_sw_cor}
    I_{V^c,\REG}(g,\Phi^c)=\Res_{s=\frac{n-m^c}{2}}E(g,s,F_\Phi).
  \end{equation}
\end{cor}
\begin{proof}
  We recall the definition of $\pi_V^{V^c}$ and $\pi_K$:
  \begin{align}\label{eq:defn_pi_QQ_pi_K}
    \pi_V^{V^c}\Phi(v^c) &= \int_{M_{r_0,n}(\A_E)}\Phi
    \begin{pmatrix}
      x \\ v^c \\0
    \end{pmatrix}
dx\\
\pi_K\Phi(v) &= \int_K \Phi(kv) dk. \nonumber
  \end{align}
Note that  $\pi_V^{V^c}$ is $\pi_Q^{Q'}$ in \cite{MR2064052}.
  To simplify notation let $\pi$ be the composite. We consider the
  local version. It is easy to see for almost all places $v$ of $k$ if
  we take $\Phi_v$ to be the characteristic function of the standard
  lattice in $V_v$ then $\pi_v\Phi_v$ is the characteristic
  function of the standard lattice in $V^c_v$.  Fix a place $v$ of $k$
  and for all other places $w$ fix $\Phi_w^{c,0} \in \cS((V_w^c)^n)$. Let
  $\Phi^c= \Phi^c_v \otimes (\otimes_{w\neq v}\Phi_w^{c,0})$. Consider the
  functional
\begin{equation}
  \ell(\Phi^c_v) := I_{V^c,\REG}(1,\Phi^c_v \otimes (\otimes_{w\neq
  v}\Phi_w^{c,0})).
\end{equation}
Here $1$ is the identity element of $G(\A)$. It is obviously
$\rU(V_v^c)$-invariant. Thus it factors through the
$\rU(V_v^c)$-coinvariant quotient of $\cS((V_v^c)^n)$.  Consider the commutative diagram
\begin{equation}
\newcommand{\rotsim}{
    \mathrel{\reflectbox{\rotatebox[origin=c]{90}{$\sim$}}}}
\xymatrix{
  &\cS(V_v^n) \ar[r]^-{\pi} &\cS((V_v^c)^n) \ar[rr]^-{\ell} \ar[rd]  \ar[rdd]_F & &\C  \\
  & & &\cS((V_v^c)^n)_{\rU(V_v^c)} \ar[ru] \ar[d]^\alpha_{\rotsim}  &  \\
  & & &R_n(V_v^c) \ar@{^{(}->}[r] &\Ind_{P_v}^{G_v}\chi_{1,v} |\ |^{\frac{m^c-n}{2}} &. } 
\end{equation}
By invariant distribution theorem, $\alpha$ is an isomorphism. It is
shown in \cite{MR2064052} that $F\circ \pi$ is $G_v$-equivariant. Also
since we assume $m^c < n$, $R_n(V_v^c)$ is irreducible. Thus $F\circ
\pi$ is a surjection. This means for any $\Phi^c_v \in \cS((V_v^c)^n)$
there exists a $\Phi_v \in \cS((V_v)^n)$ such that $F_{\Phi_v^c}=
F_{\pi\Phi_v}$. Thus $\ell\circ \pi(\Phi_v) = \ell(\Phi_v^c)$. Therefore
for any factorisable $\Phi^c\in \cS((V^c)^n(\A))$ we can find a $\Phi
\in \cS((V)^n(\A))$ such that $I_{V^c,\REG}(g,\Phi^c)=
I_{V^c,\REG}(g,\pi\Phi)$. Then the corollary follows from the previous
theorem.
\end{proof}

\section{Theta Correspondence}
\label{sec:theta_corr}

\subsection{Doubling Method}
\label{sec:doubling_method}

In this subsection we review the doubling method to prepare for the next subsection and set up some
notation. Let $X$ be an $\epsilon$-Hermitian space. It may not be
split. Let $X_a$ be as in Sec.~\ref{sec:notation} and let $X'$ be the
vector space that has the same underlying vector space as $X$ but with
the form $-\form{\ }{\ }_X$. We identify elements in $X$ and $X'$
naturally. Set $W=X\oplus X'$. Then there is a complete polarisation of
$W$ given by $W^+ = X^\Delta$ and $W^-= X^\nabla$ where
\begin{align*}
  X^\Delta = &\left\{ (x,x)\mid x \in X\right\};\\
  X^\nabla = &\left\{ (x,-x)\mid x \in X\right\}.
\end{align*}
Now consider the more general version. Let $W_a= X_a\oplus X'$. Then it has complete polarisation given by
$W_a^+=\ell_a^+ \oplus X^\Delta$ and $W_a^-=\ell_a^- \oplus X^\nabla$.


\subsection{Main Theorems}
\label{sec:main_thm}

Let $\pi$ be a cuspidal automorphic representation of $G(X)$. At this
point the additive character that figures in the Weil
representation becomes important, so it is put back in notation. We always use the character $\chi_1$ (resp. $\chi_2$)
to determine the splitting of metaplectic group over $G (X_a) (\A)$ (resp. $G
(Y_b) (\A)$). Let $\theta_{X,\psi}^Y(\pi)$ denote the theta lift of $\pi$.
The main theorems are as follows.

\begin{thm}\label{thm:theta_back_and_forth}
  Let $\pi$ be an irreducible cuspidal automorphic representation of
  $G(X)$. Assume that $\theta_{X,\psi^{-1}}^Y(\pi)$ is nonvanishing and
  cuspidal. Then
  \begin{enumerate}
  \item $\theta_{Y,\psi}^X(\chi_2^{-1}\cdot\theta_{X,\psi^{-1}}^Y(\pi))=\chi_1\pi$;
  \item $\theta_{Y,\psi}^{X_a}(\chi_2^{-1}\cdot\theta_{X,\psi^{-1}}^Y(\pi))$ is
    orthogonal to all cusp forms on $G(X_a)$ for $a>0$;
  \item $\theta_{Y,\psi}^{X_{-b}}(\chi_2^{-1}\cdot\theta_{X,\psi^{-1}}^Y(\pi))=0$ for $b>0$.
  \end{enumerate}
\end{thm}
\begin{rmk}
  Note that the theta lifts in opposite directions use additive
  characters inverse to each other. Here $\chi_1$ (resp. $\chi_2$) is
  regarded as character of $G (X) (\A)$ (resp. $G (Y)(\A)$) via $\det$. If
  we choose $\chi_1$ and $\chi_2$ to be quadratic characters then
  $\chi_i |_{\A_E^1}$ is trivial and thus $\chi_i\circ\det$ is
  trivial. Hence in this case we can leave these out of the formulae.
\end{rmk}
\begin{thm}\label{thm:theta_irred}
  Let $\pi$ be an irreducible cuspidal automorphic representation of
  $G(X)$. Assume that $\theta_{X,\psi}^Y(\pi)$ is nonvanishing and
  cuspidal. Then $\theta_{X,\psi}^Y(\pi)$ is irreducible.
\end{thm}

For the above we have also similar results in the opposite direction:
\begin{thm}
  Let $\pi$ be an irreducible cuspidal automorphic representation of
  $G(Y)$. Assume that $\theta_{Y,\psi^{-1}}^X(\pi)$ is nonvanishing and
  cuspidal. Then
  \begin{enumerate}
  \item $\theta_{X,\psi}^Y(\chi_1^{-1}\cdot\theta_{Y,\psi^{-1}}^X(\pi))=\chi_2\pi$;
  \item $\theta_{X,\psi}^{Y_a}(\chi_1^{-1}\cdot\theta_{Y,\psi^{-1}}^X(\pi))$ is
    orthogonal to all cusp forms on $G(Y_a)$ for $a>0$;
  \item $\theta_{X,\psi}^{Y_{-b}}(\chi_1^{-1}\cdot\theta_{Y,\psi^{-1}}^X(\pi))=0$ for $b>0$.
  \end{enumerate}
\end{thm}
\begin{thm}
  Let $\pi$ be an irreducible cuspidal automorphic representation of
  $G(Y)$. Assume that $\theta_{Y,\psi}^X(\pi)$ is nonvanishing and
  cuspidal. Then $\theta_{Y,\psi}^X(\pi)$ is irreducible.
\end{thm}

As pointed out in Sec.~\ref{sec:split_meta} the roles of $G(X)$ and $G(Y)$ are completely
symmetric. Thus we only need to show the theorems in one
direction. The proofs hinge on the following key
computation. Essentially the following analysis shows that
$\theta_{Y,\psi}^{X_a}(\chi_2^{-1}\theta_{X,\psi^{-1}}^Y(\pi))$ is in a space of
Eisenstein series.

By definition $\theta_{Y,\psi}^{X_a}(\chi_2^{-1}\theta_{X,\psi^{-1}}^Y(\pi))$
consists of functions of the form
\begin{equation}
  \label{eq:theta_back_and_forth}
  g_a \mapsto \int_{[G(Y)]}\int_{[G(X)]}\chi_2^{-1} (\det h)\theta_{X_a,Y,\psi}(g_a,h,\phi_1)\theta_{X,Y,\psi^{-1}}(g,h,\phi_2)f(g)dgdh
\end{equation}
for $\phi_1 \in \cS_{X_a,Y}$, $\phi_2 \in \cS_{X,Y}$ and $f\in \pi$.

We change $X$ to $X'$ which has the same underlying vector space as
$X$ but with the form $-\form{\ }{\ }_X$. We identify $X'$ with $X$
and $G(X')$ with $G(X)$ accordingly.  Then by
Prop.~\ref{prop:W_to_Wneg} to keep the action of Weil representation
`unchanged' we just need to change $\psi$ to $\psi^{-1}$. By
unfolding the definition of splitting we can check that the splittings are
compatible with this move. Set $W_a = X_a\oplus X'$. Then the integral
in \eqref{eq:theta_back_and_forth} is equal to:
\begin{align}
\label{eq:theta_back_forth_2}
    &\int_{[G(Y)]}\int_{[G(X')]}\chi_2^{-1} (\det h)\theta_{X_a,Y,\psi}(g_a,h,\phi_1) \theta_{X',Y,\psi}(g,h,\phi_2)f(g)dgdh\\
= &\int_{[G(Y)]}\int_{[G(X')]}\chi_2^{-1}(\det h) \theta_{W_a,Y,\psi}((g_a,g),h,\phi_1\otimes \phi_2)f(g)dgdh.
\end{align}
where $(g_a,g)$ denotes an element in $G (W_a) (\A)$ via obvious
identification. Now the character used for splitting over $G (Y) (\A)$ is $\chi_2^2$ and the one for $G (W_a) (\A)$  is $\chi_1$. By \eqref{eq:change_character} to use the trivial splitting for $G (Y) (\A)$ we just need to twist the action of $G (Y) (\A)$ by the character $\chi_2^{-1}$. Thus we find  the above is equal to
\begin{equation}\label{eq:theta_back_forth_4}
  \int_{[G(Y)]}\int_{[G(X')]} \theta_{W_a,Y,\psi}^{\mathrm {triv}} ((g_a,g),h,\phi_1\otimes \phi_2)f(g)dgdh
\end{equation}
where $\mathrm { triv}$ means that we are using the trivial character to determine the splitting for $G (Y) (\A)$. From now on we will drop $\mathrm {triv}$ with the trivial splitting understood.
 We would like to exchange
order of integration, so first consider the integral
\begin{align}
  &\int_{[G(X')]} I_{\REG}((g_a,g),\phi_1\otimes \phi_2) f(g) dg\\
= &\int_{[G(X')]} c_\alpha^{-1}\int_{[G(Y)]} \theta_{W_a,Y,\psi}((g_a,g),h,\omega(\alpha)(\phi_1\otimes \phi_2))f(g) dh dg.
\end{align}
 Since the inner integral, which is
$I_{\REG}((g_a,g),\phi_1\otimes \phi_2))$, is absolutely convergent we
can exchange order of integration to get
\begin{align}
  \int_{[G(Y)]} \int_{[G(X')]} c_\alpha^{-1} \theta_{W_a,Y,\psi}((g_a,g),h,\omega(\alpha)(\phi_1\otimes \phi_2))f(g) dg dh.
\end{align}
By the adjoint property of $\omega(\alpha)$ the above is equal to
\begin{align}
    \int_{[G(Y)]} \int_{[G(X')]}  \theta_{W_a,Y,\psi}((g_a,g),h,(\phi_1\otimes \phi_2))f(g) dg dh
\end{align}
which is exactly \eqref{eq:theta_back_forth_4}. When $\dim Y < \dim X + a$ we can apply
Cor.~\ref{cor:SW}. Thus \eqref{eq:theta_back_forth_4} is equal to
\begin{equation}\label{eq:theta_back_forth_3}
  \int_{[G(X')]}\Res_{s=\frac{\dim X +a -\dim Y}{2}} E((g_a,g),s,F_\varphi) f(g) dg
\end{equation}
for some $K_{G(W_a)}$-finite $\varphi \in \cS(Y^{\dim X +a}(\A))$ if
we start with $\phi_1$ and $\phi_2$ that are $K_{G(X_a)}$- and
$K_{G(X')}$-finite. Here $K_{G(W_a)}=K_{G(X_a)}\times K_{G(X')}$.

Now we compute 
\begin{equation}\label{eq:int_eis}
  \int_{[G(X')]} E^{P_a}((h,g),s,F_\varphi) f(g) dg,
\end{equation}
where $g\in G (X')$ and $h\in G(X_a)$. Note that we have added the superscript to
indicate that $E$ is an Eisenstein series associated to the Siegel
parabolic $P_a$ of $G(W_a)$ which stabilises $W_a^-$. Recall that
 $Q_a$ denotes the parabolic subgroup of $G(X_a)$ stabilising
$\ell_a^-$. 

\begin{prop}\label{prop:int_eis_is_eis}
Let $F$ be a $K_{G(W_a)}$-finite section of $\Ind_{P_a(\A)}^{G(W_a)(\A)}\chi_1 |\ |^s$. Then 
  \begin{equation}\label{eq:eis=eis}
    \int_{[G(X')]} E^{P_a}((h,g),s,F) f(g) dg = E^{Q_a}(h,s,F^f)
  \end{equation}
where 
\begin{equation}
  \label{eq:F^f}
  F^f(h,s):= \int_{G(X')(\A)}  F((h,g),s)\pi(g)f dg.
\end{equation}
\end{prop}
\begin{rmk}
  Here $F^f(h,s)$ comes from the computation below and as shown in the
  next lemma it is a section of $\Ind_{Q_a(\A)}^{G(X_a)(\A)}\chi_1 |\
  |^s\otimes \chi_1\pi$ and thus the notation of Eisenstein series on the
  right handside of \eqref{eq:eis=eis} is justified. We always use
  normalised induction.
\end{rmk}

\begin{lemma}
  The function $F^f(h,s)$ is absolutely convergent for $\Re s > (\dim X +a)/2$
 and has meromorphic continuation to the whole complex
  plane. It is a $K_{G(X_a)}$-finite section in
  $\Ind_{Q_a(\A)}^{G(X_a)(\A)}\chi_1 |\ |^s\otimes \chi_1 \pi$.
\end{lemma}
\begin{proof}
  The equation \eqref{eq:F^f} is equal to
\begin{equation}
  \int_{[G(X')]}  \sum_{\gamma \in G(X')(k)}F((h,\gamma g),s)\pi(g)f dg.  
\end{equation}
Since $P_a(k) \cap G(X')(k) =\{1\}$, we have an embedding
\begin{equation}
  G(X')(k) \hookrightarrow P_a(k) \lmod G(W_a)(k)
\end{equation}
and thus the inner sum is a partial sum of the Eisenstein series
$E^{P_a}((h,g),s,F)$. Since the Eisenstein series is absolutely
convergent for $\Re s > (\dim X +a)/2$, we find that $F^f$ is absolutely
convergent for $\Re s > (\dim X +a)/2$.

Before checking meromorphic continuation we note how $Q_a(\A)$ acts
via left translation. Let $p\in U_{Q_a}(\A)$. Then $p\in P_a(\A)$ and also
$\det_{W_a^- (\A)} p=1$. The determinant is for the action of $p$ on $W_a^- (\A)$. Thus $F^{f}(ph,s)$ is equal to
\begin{equation}
  \int_{G(X')(\A)}  F((ph,g),s)\pi(g)fdg = 
  \int_{G(X')(\A)}  F((h,g),s)\pi(g)fdg.
\end{equation}
Let $p\in M_{Q_a}(\A)$. First suppose $p$ acts trivially on $X$. Then $p$ is again in $P_a(\A)$. Then $F^{f}(ph,s)$ is equal to
\begin{equation}
  \int_{G(X')(\A)}  F((ph,g),s)\pi(g)fdg = 
  \chi_1(\det_{\ell_a^+(\A)} p)|\det_{\ell_a^+(\A)} p|^s \delta_{P_a}^{\frac{1}{2}}(p)\int_{G(X')(\A)}  F((h,g),s)\pi(g)fdg.
\end{equation}
Also we can check that we have equalities for the modular characters: $\delta_{Q_a}(p)= \delta_{P_a}(p) = |\det_{\ell_a^+} p|^{\dim X +a}$.
Secondly suppose $p|_{\ell_a^\pm(\A)}$ is trivial. Then $p \in G(X)(\A)$. Thus 
\begin{align}
  & \int_{G(X')(\A)}  F((ph,g),s)\pi(g)f dg \\
=  &\int_{G(X')(\A)}  \chi_1 (\det p) F((h,p^{-1}g),s)\pi(g)f dg \\
=  &\int_{G(X')(\A)}  \chi_1 (\det p)F((h,g),s)\pi(pg)f dg .
\end{align}
Note $(p,p)$ preserves $X^\nabla$ and as $F$ is a section of $\Ind_{P_a(\A)}^{G(W_a)(\A)}\chi_1 |\ |^s$  the first equality above holds.

Combining these we find that 
\begin{equation}
  \int_{G(X')(\A)}  F((\cdot,g),s)\pi(g)f dg
\end{equation}
is an element in the $\Ind_{Q_a(\A)}^{G(X_a)(\A)} (\chi_1 |\ |^s \otimes \chi_1\pi)$. It is obviously $K_{G(X_a)}$-finite since $F(s)$ is $K_{G(W_a)}$-finite.

Now we check that we have meromorphic continuation. Since $G(X_a)(\A)=
Q_a(\A)K_{G(X_a)}$ and $F(s)$ is $K_{G(W_a)}$-finite, we may just
assume that $h$ is in $Q_a(\A)$. Since  
$U_{Q_a}(\A)$ acts trivially on $F$ by left translation and the $\GL(\ell_a^+)(\A_E)$-part of
$M_{Q_a}(\A)$ acts by $\chi_1 |\ |^s$ we may further assume that $h$ is in
the subgroup $G(X)(\A)$ of $G(X_a)(\A)$. Then $F^f(s)$ restricted to
$G(X)(\A)$ is in the space of $\chi_1\pi$. Consider for all cusp form
$\xi \in \chi_1\pi$ the $L^2$-inner product
 \begin{equation}\label{eq:L2_inner}
   \int_{[G(X')]} F^f(h,s)\overline{\xi(h)} dh.
 \end{equation}
 If it has meromorphic continuation to the whole complex plane for all
 cusp form $\xi \in \chi_1\pi$, then $F^f(s)$ has meromorphic
 continuation. The equation \eqref{eq:L2_inner} by definition is equal
 to
\begin{align}
  &\int_{[G(X')]}  \int_{G(X')(\A)}  F((h,g),s)f(g)  \overline{\xi(h)} dg dh\\
=& \int_{[G(X')]}  \int_{G(X')(\A)} F((1,h^{-1}g),s)f(g)  \overline{\xi(h)} dg dh\\
=& \int_{[G(X')]}  \int_{G(X')(\A)} F((1,g),s)f(hg)  \overline{\xi(h)} dg dh\\
=&  \int_{G(X')(\A)} F((1,g),s)\form{f(g)}{  \xi}_{L^2} dg. 
\end{align}
This according to the basic identity in
\cite{piatetski-shapiro87:_l_funct_for_class_group} is equal to
\begin{equation}
   \int_{[G(W)]} E^{P}((g_1,g_2),s,F)f(g_2)  \overline{\xi(g_1)} dg_1 dg_2
\end{equation}
where $P$ is the Siegel parabolic of $G(W)$ stabilising $W^-$ and we
restrict $F$ to $G(W)(\A)$ so it is in the induced representation
$\Ind_{P(\A)}^{G(W)(\A)}\chi_1 |\ |^{s+\frac{a}{2}}$. Thus
\eqref{eq:F^f} has meromorphic continuation.
\end{proof}

\begin{proof}[proof of Prop.~\ref{prop:int_eis_is_eis}]
  The method of proof is a generalisation of the basic identity in
  \cite{piatetski-shapiro87:_l_funct_for_class_group}. We unfold and
  study the double coset $P_a \lmod G(W_a) / G(X_a)\times G(X')$ and
  then identify the negligible orbits and the main orbit.

  Consider the set $P_a \lmod G(W_a)$ that parametrises the maximal
  isotropic subspaces of $W_a$ over $E$. Let $L$ be a maximal
  isotropic subspace. Let $d=\dim (L\cap X')$. Then $\dim (L \cap X_a)
  = d+a$. The proof in
  \cite[Lemma~2.1]{piatetski-shapiro87:_l_funct_for_class_group} goes
  through word for word. The $G(X_a)\times G(X')$-orbits of maximal
  isotropic subspaces are parametrised by the invariant $d$. When
  $d=0$, one representative for the double coset is $W_a^+$. For each
  $d=0,\ldots, \dim X$, take $g_d$ to be the element in $G(W_a)$ that
  conjugates to $W_a^+$ a representative of the orbit corresponding to
  $d$ . For clarity of notation we omit taking $k$-points in the sums
  below. For $\Re s$ large the left hand side of \eqref{eq:eis=eis}
  is equal to
\begin{align*}
  &\int_{[G(X')]} \sum_{\gamma \in P_a\lmod G(W_a)} F(\gamma(h,g),s)f(g)dg\\
= &\int_{[G(X')]} \sum_{d=0}^{\dim X} \sum_{\gamma\in   g_d^{-1} P_a g_d \cap G(X_a) \times G(X') \lmod G(X_a) \times G(X')}F(g_d\gamma(h,g),s)f(g)dg\\
= & \sum_{d=0}^{\dim X} \int_{[G(X')]}  \sum_{\delta \in g_d^{-1} P_a g_d G(X')\cap G(X_a)\lmod G(X_a)}\sum_{\gamma \in g_d^{-1} P_a g_d \cap G(X') \lmod G(X')} F(g_d(\delta,\gamma)(h,g),s)f(g)dg\\
= & \sum_{d=0}^{\dim X} \int_{g_d^{-1} P_a(k) g_d \cap G(X')(k) \lmod G(X')(\A)}  \sum_{\delta \in g_d^{-1} P_a g_d G(X')\cap G(X_a)\lmod G(X_a)}  F(g_d(\delta h,g),s)f(g)dg\\
= & \sum_{d=0}^{\dim X} \int_{g_d^{-1} P_a(\A) g_d \cap G(X')(\A) \lmod G(X')(\A)} 
\int_{[g_d^{-1} P_a g_d \cap G(X')]} \\
&\qquad \sum_{\delta \in g_d^{-1} P_a g_d G(X')\cap G(X_a)\lmod G(X_a)} 
 F(g_d(\delta h,gg'),s)f(gg')dg dg'.
\end{align*}

We will verify as in
\cite{piatetski-shapiro87:_l_funct_for_class_group} that only the
integral corresponding to $d=0$ contributes. Fix $d$ and thus the
corresponding maximal isotropic subspace is $L=W_a^+g_d$. Let $Q_d$ be
the parabolic subgroup of $G(X')$ stabilising the isotropic subspace
$L_1:=L \cap X'$ and let $U_d$ be its unipotent radical. Note that we
view $G(X')$ as a subgroup of $G(W_a)$. Suppose $d>0$. Then $Q_d$ is a
proper parabolic subgroup. Decompose $X'$ as $L_1\oplus X''$.  Let
$n\in U_d$ and let $x\in L$. Write $x=x_1+x_2+x_3$ where $x_1 \in
X_a$, $x_2 \in L_1$ and $x_3 \in X''$. Then $xn=x_1 + x_2 +x_3 +x_4
=x+x_4$ for some $x_4 \in L_1$. Thus $xn\in L$. We have shown that
$U_d$ stabilises $L$, that $U_d \subset g_d^{-1} P_a(k) g_d \cap
G(X')(k)$ and that $\det_{L} n =1$. Thus $F(g_d(\delta h,g),s)$ is
invariant when $g$ is changed to $ng$. Then the inner integral can be
further split into two iterated integrals with an inner integral being
over $[U_d]$:
\begin{equation}
  \int_{[U_d]} f(ng)dn
\end{equation}
which is zero by cuspidality of $f$. Thus when $d>0$ these are indeed negligible orbits.

Now consider the main orbit corresponding to $d=0$. Since $ P_a(k)\cap G(X')(k) = \{1\}$ the contribution is
\begin{equation}\label{eq:main_orbit}
   \int_{ G(X')(\A)}  \sum_{\delta \in  P_a  G(X')\cap G(X_a)\lmod G(X_a)}  F((\delta h,g),s)f(g)dg.
\end{equation}

We claim that $P_a G(X')\cap G(X_a) = Q_a$. Here $Q_a$ is from our old
notation: the parabolic subgroup of $G(X_a)$ stabilising
$\ell_a^-$. It is easy to see $Q_a \subset P_a G(X')\cap G(X_a)$. Now
suppose $h \in P_aG(X')\cap G(X_a)$. Then $(h,1)=p(1,g)$ for some
$p\in P_a$ and $g\in G(X')$. We have $(h,g^{-1})\in P_a$ and thus
$(hg,1) \in P_a$ where we view $g\in G(X')$ as an element in the subgroup $G(X)$ of
$G(X_a)$ via the obvious embedding. In other words $hg$ is  contained in $P_a \cap G(X_a)$ which is isomorphic
to the group $\GL(\ell_a^+)\times U_{Q_a}$. Thus $h \in
\GL(\ell_a^+)\times U_{Q_a} \times G(X') = Q_a$.

With this we find \eqref{eq:main_orbit} is equal to
\begin{equation}
  \sum_{\gamma \in Q_a(k)\lmod G(X_a)(k)} F^f( \gamma h)
\end{equation}
for $h \in G(X_a)(\A)$ as required.
\end{proof}

It follows from the proposition above and the analysis before it (c.f. \eqref{eq:theta_back_forth_3}) we have
\begin{prop}\label{prop:theta_is_Eis}
  For $\dim Y < \dim X +a$, $\theta_{Y,\psi}^{X_a}\chi_2^{-1}\theta_{X,\psi^{-1}}^{Y}(\pi)$ is contained in the space 
  \begin{equation}
    \{ \Res_{s=\frac{1}{2}(\dim X +a -\dim Y)} E^{Q_a}(\cdot,s,F^f ) | f\in \pi \}
  \end{equation}
  where $F^f$ is defined in \eqref{eq:F^f}.
\end{prop}

Now we try to remove the condition on $a$ in the previous
proposition. The idea is to take constant terms on the Eisenstein
series and on the theta lifts. Let $Q_{A-a}$ denote the parabolic
subgroup of $G(X_A)$ stabilising $\ell_{A-a}^-$ which is a subspace of
$\ell_A^-$ consisting of $A-a$ hyperbolic planes. Temporarily let $\chi_1$ be subsumed into $\pi$. Then on the
Eisenstein side we have

\begin{lemma} \label{lemma:const_term_of_eis}Consider the constant term of $E^{Q_A}(g,s,f)$ along
  $Q_{A-a}$ for $f\in \Ind_{Q_A(\A)}^{G(X_A)(\A)} (\chi_1 |\ |^s \otimes
  \pi)$. If we consider it as a function of $g\in G(X_a)$ then its residue at $s^0=\frac{1}{2}(\dim X +A -\dim Y)$ is
  \begin{enumerate}
  \item  orthogonal to all cusp forms on $G(X_a)$ for $a>0$;
  \item   in $\pi$ for $a=0$;
  \item zero for $a<0$.
  \end{enumerate}
\end{lemma}
\begin{proof}
  We compute the constant term. Since $\chi_1 |\ |^s \otimes \pi$ may
be non-cuspidal, we go to the cuspidal support. Let $Q_A'$ be the standard
parabolic subgroup of $G(X_A)$ whose Levi is isomorphic to
$(\Res_{E/k}\GL_1)^A\times G(X)$ where $\Res_{E/k}$ is Restriction of scalar
of Weil. 
There exists a section $F$ of
\begin{equation}
  \Ind_{Q_A'(\A)}^{G(X_A)(\A)} \chi_1|\ |^{s_1}\times\cdots\times \chi_1|\ |^{s_A} \times \pi
\end{equation}
such that 
\begin{equation}
  E^{Q_A}(s,f) = \prod_{i=1}^{A} (s_i - s - \frac{A-2i+1}{2}) E^{Q_A'}(s_1,\ldots,s_A,F)|_{s_i = s + \frac{A-2i+1}{2}}.
\end{equation}
Thus we need to compute $[E^{Q_A'}(s_1,\ldots,s_A,F)]_{Q_{A-a}}$. We
consider the double cosets $Q_A' \lmod G(X_A) / Q_{A-a}$. Let $\Omega$
be the set of Weyl elements $w$ such that $w$ is of minimal length in
the double coset $Q_{A-a} w Q_A'$. If we identify $\Omega$ with the group of signed permutations which is a
subset of maps from $\{1,\ldots, A\}$ to $\{\pm 1,\ldots,\pm A\}$, then $\Omega$ is the set of maps
\begin{equation}
  \left\{w \middle| \substack{w^{-1}(i) < w^{-1}(j) \text{ if $1\le i < j \le A-a$ or $A-a+1\le i < j \le A$ }\\
 w^{-1}(i) >0 \text{ if $A-a+1 \le i \le A$} }\right\}.
\end{equation}

Set $\underline{s}=(s_1,\ldots,s_A)$ and
$w\underline{s}=(s_{w^{-1}(1)},\ldots,s_{w^{-1}(A)})$. Let $g$ be an element in the Levi $M_{A-a}(\A)$ of $Q_{A-a}(\A)$.
 Then $[E^{Q_A'}(g,
\underline{s},F)]_{Q_{A-a}}$ is equal to
\begin{equation}\label{eq:sum_of_eis}
  \sum_{w \in \Omega} E^{w Q_A'w^{-1} \cap M_{A-a}}_{M_{A-a}}(g,w\underline{s},M(w,\underline{s})F)
\end{equation}
where $M(w,\underline{s})$ is the intertwining operator associated to
$w$. Here $E^{w Q_A'w^{-1} \cap M_{A-a}}_{M_{A-a}}$ is an Eisenstein
series on $M_{A-a}$ with respect to the parabolic subgroup $w
Q_A'w^{-1} \cap M_{A-a}$. We restrict to $g\in G(X_a)(\A)$. Then in
the cone of absolute convergence \eqref{eq:sum_of_eis} is equal to
\begin{align}
  &\sum_{w \in \Omega} \sum_{\gamma\in w Q_A'w^{-1}\cap \Res_{E/k} \GL_{A-a} \lmod \Res_{E/k}\GL_{A-a}} E^{w Q_A'w^{-1} \cap G(X_a)}_{G(X_a)}(\gamma g,w\underline{s},M(w,\underline{s})F).
\end{align}
If we integrate this against a cusp form on $G(X_a)$ the integral will
vanish. This proves part (1).

Now let $a=0$. Then the residue in question becomes
\begin{equation}
 \prod_{i=1}^{A} (s_i - s - \frac{A-2i+1}{2}) (s- s^0)  \sum_{w \in \Omega}  \sum_{\gamma\in w Q_A'w^{-1}\cap \Res_{E/k}\GL_{A} \lmod \Res_{E/k}\GL_{A}} 
M(w,\underline{s})F(\gamma g, w\underline{s}).
\end{equation}
evaluated first at $s_i = s + \frac{A-2i+1}{2}$ and then at $s= s^0$. If
we restrict $g$ to the subgroup $G(X)(\A)$ of $G(X_A)(\A)$ then this is in the
space of $\pi$. Thus we have proved part (2).

Let $b>0$. For $[E^{Q_A'}(g,s_1,\ldots,s_A,F)]_{Q_{A+b}}$ we take constant term in steps: 
\begin{equation}
  [[E^{Q_A'}(g,s_1,\ldots,s_A,F)]_{Q_{A}}]_{Q_{A+b}\cap G(X_{-b})}.
\end{equation}
The residue of the first step falls in the space of $\pi$ which is assumed to be
cuspidal. Thus the second step gives $0$. Thus we have proved part
(3).
\end{proof}

With this preparation we are ready to begin
\begin{proof}[Proof of Theorem~\ref{thm:theta_back_and_forth}]
  We take $A$ large so that $\dim X +A >\dim Y$ and thus we can apply
  Prop.~\ref{prop:theta_is_Eis}. For the theta lift side, for
  any $a$ we have the tower property (Prop.~\ref{prop:tower_property})
  \begin{equation}
  [\theta_{Y,\psi}^{X_A}\chi_2^{-1}\theta_{X,\psi^{-1}}^{Y}(\pi)]_{Q_{A-a}}|_{G(X_a)}=\theta_{Y,\psi}^{X_a}\chi_2^{-1}\theta_{X,\psi^{-1}}^{Y}(\pi).
  \end{equation}

  Suppose $0<a<A$. We take constant terms along $Q_{A-a}$ on the
  Eisenstein side of the spaces and then restrict to $G(X_a)(\A)$.  By
  Lemma~\ref{lemma:const_term_of_eis}, the constant term is orthogonal
  to cusp forms on $G(X_a)$. Thus we have proved Part (2).  We then
  take constant term along $Q_A$. The residue of Eisenstein series
  falls in the space of $\chi_1\pi$. (We unsubsume the $\chi_1$.) Thus we get part (1). Now take constant
  term along $Q_{A+b}$ for $b>0$. The residue of Eisenstein series
  becomes $0$. Thus we get part (3).
\end{proof}

Finally for the irreducibility result:
\begin{proof}[Proof of Theorem~\ref{thm:theta_irred}]
  Suppose $\sigma$ is an irreducible submodule of
  $\theta_{X,\psi^{-1}}^Y(\pi)$. By
  Thm.~\ref{thm:theta_back_and_forth} the subspace
  $\theta_{Y,\psi}^{X_a}(\chi_2^{-1}\sigma)$ of
  $\theta_{Y,\psi}^{X_a}\chi_2^{-1}\theta_{X,\psi^{-1}}^{Y}(\pi)$ is
  orthogonal to cusp forms on $G(X_a)$ and also
  $\theta_{Y,\psi}^{X_{-b}}(\chi_2^{-1}\sigma)=0$. Thus by the
  cuspidality of first occurrence we must have that
  $\theta_{Y,\psi}^X(\chi_2^{-1}\sigma)$ is nonvanishing and
  cuspidal. Since $\theta_{Y,\psi}^X(\chi_2^{-1}\sigma) \subset
  \theta_{Y,\psi}^{X}\chi_2^{-1}\theta_{X,\psi^{-1}}^{Y}(\pi) =
  \chi_1\pi$ and $\pi$ is irreducible, we find
  $\theta_{Y,\psi}^X(\chi_2^{-1}\sigma)=\chi_1\pi$. Then by theta
  lifting in the other direction we find $\sigma =
  \theta_{X,\psi^{-1}}^{Y}(\pi)$. Thus $\theta_{X,\psi^{-1}}^{Y}(\pi)$
  is indeed irreducible.  
\end{proof}

\section*{Acknowledgement}
The author would like to thank Professor Dihua Jiang for numerous
helpful suggestions during the preparation of this work. Also the
author would like to thank Professor Christian Kaiser for clarifying
the concept of metaplectic group and the referees for pointing out the
inconsistency of splitting in the preliminary version. The author is
grateful to University of Minnesota where the bulk of this work was
written and Max-Planck Institute for Mathematics for its inspiring
atmosphere in which this work is finalised.

\bibliographystyle{plain}
\bibliography{MyBib}{}

\end{document}